\documentclass[a4paper, 12pt]{amsart}

\newtheorem{thm}{Theorem}

\usepackage{lipsum}

\usepackage{url}
\usepackage{mathtools}
\usepackage[hidelinks]{hyperref}
\usepackage[capitalize,noabbrev]{cleveref}
\usepackage{tikz}
\usetikzlibrary {decorations.pathmorphing}
\usetikzlibrary{calc}
\usepackage{color}
\usepackage{centernot}
\usepackage{enumitem}
\usepackage[utf8]{inputenc}
\usepackage{amsmath,amsthm,amssymb}
\usepackage[margin=.75in]{geometry}
\usepackage{tabularx}

\usepackage{ytableau}
\usepackage{todonotes}

\DeclareMathOperator{\Des}{\mathrm{D}}
\DeclareMathOperator{\st}{\,|\,}
\DeclareMathOperator{\comaj}{\mathrm{comaj}}
\DeclareMathOperator{\SYT}{\mathrm{SYT}}
\DeclareMathOperator{\Lin}{\mathcal{L}}
\DeclareMathOperator{\ddeg}{\mathrm{ddeg}}
\DeclareMathOperator{\Z}{\mathbb{Z}}
\DeclareMathOperator{\krew}{Krew}
\DeclareMathOperator{\bfc}{{\bf c}}

\newcommand{\qbinom}{\genfrac{[}{]}{0pt}{}}
\newcommand{\rect}[2]{#1 \times #2}
\newcommand{\RPP}{\mathcal{RPP}}

\newtheorem{question}[thm]{Question}
\newtheorem{remark}[thm]{Remark}
\newtheorem{proposition}[thm]{Proposition}
\newtheorem{corollary}[thm]{Corollary}
\newtheorem{example}[thm]{Example}

\newtheorem{lemma}[thm]{Lemma}

\DeclareMathOperator{\motz}{Motz}
\DeclareMathOperator{\motzE}{Motz^1}
\DeclareMathOperator{\motzT}{Motz^2}
\DeclareMathOperator{\motzET}{Motz^{1,2}}
\DeclareMathOperator{\bal}{Bal}

\DeclareMathOperator{\cat}{Cat}

\newcommand{\bij}{\alpha}
\newcommand{\bijTwo}{\beta}

\title[Set-Valued Catalan Combinatorics]{Set-Valued Catalan Combinatorics}

\author[Lazar and Linusson]{Alexander Lazar and Svante Linusson}


\usepackage{biblatex}
\begin{document}

\begin{abstract}Set-valued standard Young tableaux are a generalization of standard Young tableaux due to Buch (2002) with applications in algebraic geometry. The enumeration of set-valued SYT is significantly more complicated than in the ordinary case, although product formulas are known in certain special cases. In this work we study the case of two-rowed set-valued SYT with a fixed number of entries. These tableaux are a new combinatorial model for the Catalan, Narayana, and Kreweras numbers, and can be shown to be in correspondence with both $321$-avoiding permutations and a certain class of bicolored Motzkin paths. We also introduce a generalization of the set-valued comajor index studied by Hopkins, Lazar, and Linusson (2023), and use this statistic to find seemingly new $q$-analogs of the Catalan and Narayana numbers.\end{abstract}
\maketitle

\section{Introduction}

\subsection{Set-Valued Tableaux}
Let $\lambda \vdash n$. A \emph{set-valued Young tableau} of shape $\lambda$ is a filling $S$ of the cells of the Ferrers diagram of $\lambda$ with nonempty sets of positive integers. They were introduced by Buch \cite{buch2002littlewood} to study the $K$-theory of the Grassmannian, and have since appeared in both algebro-geometric and combinatorial contexts (see, inter alia, \cite{anderson2017kclasses, chan2021euler, chan2018genera, drube2018setvalued, kim2020enumeration,reiner2018poset}). 

A set-valued Young tableau is \emph{standard} if:
\begin{enumerate}
\item The sets in the cells of $\lambda$ form a set partition of $[n+k]$ for some $k\geq 0$, and
\item If $u$ is (weakly) northwest of $v$ in $\lambda$ then $\max S(u) < \min S(v)$.
\end{enumerate}

We write $\SYT^{+k}(\lambda)$ for the set of set-valued standard Young tableaux of $\lambda$ with entries in $[n+k]$. Intuitively, a set-valued SYT $S$ can be thought of as an integer filling of $\lambda$ (filling each cell $u$ with $\min S(u)$) along with $k$ \emph{extra elements}. The combinatorics of these objects is much more intricate than in the ordinary case; there is no known analog of the hook-length formula for counting set-valued SYT in general, although Anderson, Chen, and Tarasca \cite{anderson2017kclasses} proved a determinantal formula for counting them.

For the purposes of enumerating the elements of $\SYT^{+k}(\lambda)$, it is sometimes useful to view a set-valued tableaux $S$ from a different perspective.

\begin{lemma}\label{SVSYTTriple}
A standard set-valued Young tableau of shape $\lambda$ is equivalent to the following data:
\begin{enumerate}
\item A standard Young tableau $S^*$ of shape $\lambda$,
\item A weak chain $\lambda^{\bullet}$ of subshapes 
$$\emptyset = \lambda_0 \subsetneq \lambda_1 \subseteq \cdots \subseteq \lambda_k \subseteq \lambda_{k+1} = \lambda$$
\item A choice of a corner cell $u_i$ of $\lambda_i$ for each $1\leq i \leq k$.
\end{enumerate}
\end{lemma}

For now, consider the following illustrative example (a proof will be given in Section \ref{Poset-Section}).

\begin{example}\label{SVSYTEx}
Consider the following set-valued SYT $T \in \SYT^{+4}(3\times 4)$:
\begin{center}
$\begin{array}{|c|c|c|c|}
\hline
1 & 2 & 7 & 8\\
\hline
3 & 4, 5 & 11 & 13\\
\hline
6, 9, 10 & 12 & 14, 15 & 16\\
\hline
\end{array}$
\end{center}
There are cells with extra entries at matrix coordinates $(2,2)$, $(3,1)$, and $(3,3)$. Among these, the cell at $(2,2)$ has the smallest extra entry: $5$. We define $\lambda_1$ to be the subshape of $3\times 4$ for which the entries of $T$ are between $1$ and $5$:
\begin{center}
\begin{ytableau}
*(yellow) & *(yellow) & & \\
*(yellow) & *(yellow) \star & & \\
& & &
\end{ytableau}
\end{center}
The starred cell at matrix position $(2,2)$ is $u_1$.
The next extra entry is $9$, at matrix position $(3,1)$. We define $\lambda_2$ to be the subshape for which the entries of $T$ are between $1$ and $9$:
\begin{center}
\begin{ytableau}
*(yellow) & *(yellow) & *(blue) & *(blue)\\
*(yellow) & *(yellow) & & \\
*(blue) \star & & &
\end{ytableau}
\end{center}
Then since $9$ belongs to the starred cell $(3,1)$, we define that cell to be $u_2$. The next extra entry is $10$, which is in the same cell as the extra entry $9$. Then $\lambda_3 = \lambda_2$ and $u_3 = u_2$.

The last extra entry is $15$, at matrix position $(3,3)$. We have that $\lambda_4$ is the subshape $(4,4,3)$ consisting of the cells of $T$ whose entries are between $1$ and $15$.
\begin{center}
\begin{ytableau}
*(yellow) & *(yellow) & *(blue) & *(blue)\\
*(yellow) & *(yellow) & *(red) & *(red) \\
*(blue)  & *(red) &*(red) \star &
\end{ytableau}
\end{center}
Since $15$ belongs to the cell at position $(3,3)$, we define $u_4$ to be that corner cell. Finally, $\lambda_5$ is the entire shape.

We obtain $T^*$ from $T$ by removing the $4$ extra entries from $T$ and decrementing the remaining entries of the cells in $\lambda_i\setminus \lambda_{i-1}$ by $i-1$ for each $i$:
\begin{center}
\begin{ytableau}
1 & 2 & 6 & 7\\
3 & 4 & 8 & 10\\
5 & 9 & 11 & 12
\end{ytableau}
\end{center}
The entries in the yellow cells are decremented by $0$; those in the blue cells are decremented by $1$; those in the red cells are decremented by $3$ (notice there are no cells in $\lambda_3\setminus \lambda_2$); and the entry of the bottom right square is decremented by $4$.
\end{example}

This construction allows us to define a version of the \emph{comajor index} for set-valued tableaux.
Let $S \in \SYT^{+k}(\lambda)$, and decompose $S$ into $k$ skew tableaux $Y_1,\dots, Y_{k}$ (where $T_j$ has shape $\lambda_j\setminus \lambda_{j-1}$) and $k$ additional elements $x_1,\dots,x_k$ as in \cref{SVSYTEx}. A \emph{(natural) descent} of $Y_i$ is an entry $j$ of $Y_i$ such that $j+1$ is also an entry of $Y_i$ and is in a higher row\footnote{This is different from the usual definition of a descent in a Young tableau; our definition instead comes from the theory of $P$-partitions.}.

We write $\Des(Y_i)$ for the descent set of $Y_i$, and we define the \emph{set-valued descents} of $S$ to be $$\Des^{+k}(S) \coloneqq \bigsqcup D(Y_i) \sqcup \{x_1,\dots, x_k\}.$$

The \emph{set-valued comajor index} of $S$ is then defined as $$\comaj^{+k}(S) \coloneqq \displaystyle \sum_{x \in \Des^{+k}(S)}(n+k - x).$$

\begin{example}
Continuing from \cref{SVSYTEx}: 
$$D^{+4}(S) = \{6\}\sqcup\{12\}\sqcup\{5,9,10,15\},$$ so $\comaj^{+4}(S) = 38$.
\end{example}

The $k=1$ version set-valued comajor index was recently used by Hopkins, Lazar, and Linusson \cite{hopkins2021qenumeration} to find a product formula for $\displaystyle\sum_{S \in \SYT^{+1}(a\times b)}q^{\comaj^{+1}(S)}$ analogous to Stanley's hook-content formula. Our generalized version is motivated by the probabilistic reasoning used in \cite{hopkins2021qenumeration} --- when one attempts to extend these arguments to general $\SYT^{+k}$, the $\comaj^{+k}$ statistic emerges quite naturally and yields extensions of some of the results of \cite{hopkins2021qenumeration} to the general case. 

\subsection{Main Results}
The present work considers set-valued SYT from a different perspective. Rather than fixing the shape and the number of extra elements of a set-valued tableau $S$, we instead fix the number of rows and total number of elements.

This change of perspective has proven to be fruitful; if we restrict our attention to the case of two-row tableaux and fix the total number of elements while letting the number of columns vary, we obtain several new results:
\begin{itemize}
    \item For fixed $n$ and $i$, exact counts of $\bigsqcup_{2b-i+k = n}\SYT^{+k}(b,b-i)$ for all $0\leq i \leq b$. For $i=0$, that is, rows of equal length, it is the Catalan number (Equation~\eqref{eq:sv_syt_cat}) and for general $i$ it is a ballot number plus a binomial coefficient (Theorem~\ref{BallotThm}). 
    \item New models for the Catalan (Proposition \ref{prop:321_bij}), Narayana, and Kreweras (Corollary \ref{cor:321_bij2}) numbers (proved via a bijection with $321$-avoiding permutations).
    \item A new summation formula for the $321$-avoiding permutations by the number of peaks (Corollary \ref{321PeaksCor}).
    \item Exact counts of several families of lattice paths arising from these tableaux (Proposition \ref{MotzProp} and Theorem \ref{BallotThm}).
    \item Generating function formulas for weighted counts of several of these families (Section~\ref{Sec:GFun}).
    \item Seemingly-new families of $q$-Catalan and $q$-Narayana numbers (Section \ref{qCatSect}).
\end{itemize} 

Moreover, our study of the $q$-enumeration of the two-rowed case leads us to what we believe is the ``correct'' $q$-statistic for set-valued tableaux ($\comaj^{+k}$), about which we prove several general results:
\begin{itemize}
\item A probabilistic interpretation of $\comaj^{+k}$ which expresses the expected value of a simple statistic as a ratio between the $q$-count of set-valued SYT and the $q$-count of ordinary SYT (Theorem~\ref{thm:ExpectedComaj}).
\item An equidistribution result about the descent sets of elements of $\SYT^{+k}(\lambda)$ and $\SYT^{+k}(\lambda')$ (Theorem~\ref{thm:Equidistribution}).
\end{itemize}

Many of these results were announced (without proof) in the extended abstract \cite{PlusKFPSAC}.

\section*{Acknowledgements} 
Both authors were supported by grant 2018-05218 from VR the Swedish Science Council. The first author was supported by ARC grant ``From Algebra to Combinatorics, and Back''. The second author also by 2022-03875 from VR. The results of Section~\ref{Sec:GFun} were originally conjectured by the authors in their extended abstract for FPSAC 2024. We wish to thank Guillaume Chapuy and Pierre Bonnet for suggesting the methods used to prove them.

\section{Bijection to 321-avoiding Permutations}

In this section we will use a bijection to 321-avoiding permutations of length $n-1$ to prove that for any $n\ge 2$
\begin{equation} \label{eq:sv_syt_cat}
 \sum_{2b+k = n} \#\SYT^{+k}(\rect{2}{b}) = \cat(n-1).
\end{equation}
 A permutation $\pi=\pi_1\ldots\pi_n$ is called \emph{321-avoiding} if it does not have three elements $\pi_i>\pi_j>\pi_k$ for $ 1\le i<j<k\le n$. Another well-known way to describe 321 avoiding permutations is as follows. Recall that  a \emph{right-to-left minimum} in a permutation $\pi$ is an element $\pi_i$ such that $\pi_i<\pi_j$ for all $j>i$. The right-to-left minima of any permutation form an increasing sequence (when read from the left). The condition that a permutation is 321-avoiding is equivalent to asking that the elements 
that are not right-to-left minima  also form an increasing sequence. This characterization dates back to the early 1900s; see~\cite[Vol. I, Section V, Chapter III]{macmahon1915combinatory}. \footnote{The text considers $123$-avoiding permutations, which are the reverses of $321$-avoiding permutations.}
 Visualising the permutation with a permutation matrix, the right-to-left minima  will be on or below the main diagonal and the other elements above the diagonal. Forming a lattice path around the elements on or above the diagonal gives a direct bijection to south-east lattice paths above the diagonal, which is one of the many standard representations of Catalan objects. Alternatively, one can draw a lattice path below the right-to left minima and then rotate the drawing by a half turn. We also need to define an \emph{inner valley}\footnote{An inner valley differs from an ordinary valley in that neither the first nor the last position can be an inner valley.} in a permutation $\pi\in S_n$ as an element $\pi_j, 1<j<n$ such that $\pi_{j-1}>\pi_j<\pi_{j+1}$. 

\begin{example}\label{321Ex} The permutation $\pi = \ 3\  5 \ \overline 1\  \overline 2\  7\  8\ \overline 4\ {10}\ {11}\ \overline 6\ \overline9$ is $321$-avoiding. We overline the right-to-left-minima.
\end{example}
\begin{figure}
\begin{center}
    \includegraphics{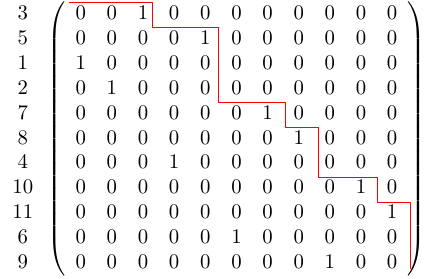}
\end{center}
\caption{The permutation matrix of $\pi$ from Example~\ref{321Ex}, along with its associated lattice path.}
\end{figure}

For a fixed $n$ we now define a map 
$$\bij: \displaystyle\bigsqcup_{2b+k = n} \SYT^{+k}(\rect{2}{b}) \mapsto \{\text{$321$-avoiding permutations of $[n-1]$}\}.$$ 
Let $T\in \SYT^{+k}(\rect{2}{b})$, with $k+2b=n$.
\begin{itemize}
\item [(1)] Remove the largest element $n$ from $T$, so it contains the numbers from 1 to $n-1$.
\item [(2)] The permutation $\bij(T)$ starts with all except the largest element in the top left box, followed by the entries of the box directly below it, and then the largest element of the top left box. The permutation 
continues with the elements in the second box in the top row except the largest, then all elements in the box below it, then the largest element in the second box in the top row. We continue in this way, placing the elements of the $i$th box from the left in the bottom row immediately before the largest element of the $i$th box in the top row.
\end{itemize}

\begin{example}
\ytableausetup{boxsize=2.7em}
\[
\begin{ytableau} 
 1,2 & 3,4,6 & 7 &10 \\ 5, 8& 9 & 11,12 &{13,14} 
 \end{ytableau}
\quad
\overset {\bij}\mapsto
 \pi= \overline 1 \ 5\   8\ \overline 2 \ \overline 3 \ \overline 4\ 9\   \overline 6\ 11\  12\  \overline 7\ 13\ \overline {10}
  \]
  \end{example}

The resulting permutation $\bij(T)$ will by construction have the numbers in the top row as its right-to-left minima. The elements in the bottom row (except $n$, which has been deleted) will form another increasing sequence. The permutation formed is thus $321$-avoiding. Note that the largest elements in the top boxes in columns $1,\dots, b-1$ will be inner valleys in the permutation and there are no other inner valleys. \footnote{We thank FindStat \cite{FindStat}, which helped us find that the refinement into columns was equidistributed with number of inner peaks. This equidistribution was a key insight into finding the bijection $\bij$.}

The inverse of $\bij$ is reasonably simple. Given a $321$-avoiding permutation $\pi$, we mark the right-to-left minima. They will be top elements. The inner valleys, which by definition are right-to-left minima, will be maximal elements in one box each. The other top elements will go into the same box as the inner valley closest to the right and the top elements to the right of rightmost inner valley (there must be such elements) will be the content of one more top box. So the number of columns will be one more than the number of inner valleys.  
Consecutive elements that are not right-to-left minima are placed in the same box in the bottom row and $n$ is added to the last box. The last box may be void of other elements, which happens exactly when the last two positions of $\pi$ do not form a descent. 

\begin{example} We reuse the permutation from Example~\ref{321Ex} to illustrate $\bij^{-1}$.
\ytableausetup{boxsize=2.3em}
\[
\pi = \ 3\  5 \ \overline 1\  \overline 2\  7\  8\ \overline 4\ {10}\ {11}\ \overline 6\ \overline {9}
\overset {\bij^{-1}}\mapsto
\quad
\begin{ytableau} 
 1 &2,4 & 6 &9 \\ 3,5& 7,8 & 10,11 &12
 \end{ytableau}
\]\end{example}

We summarize some basic properties of $\bij$. 
\begin{thm} \label{prop:321_bij} For all $n\geq 2$:
\begin{enumerate}
\item The map $\alpha$ is a bijection from $\displaystyle\bigsqcup_{2b+k = n} \SYT^{+k}(\rect{2}{b})$ to the set of $321$-avoiding permutations of $[n-1]$.
\item The elements in the top row of $T$ form the sequence of right-to-left minima in $\alpha(T)$.
\item If $T$ has $b$ columns, then $\alpha(T)$ will have $b-1$ inner valleys.
\end{enumerate} 
\end{thm}

Recall from the theory of Catalan numbers that the number of Dyck paths of length $2n$ is counted by the Catalan number $\cat(n)$. The number of such paths with $m$ peaks
is the Narayana number $N_{n,m}=\frac{1}{m}\binom{n}{m-1}\binom{n-1}{m-1}$. 

There is even one further refinement. Let $c_i$ be the number of upsteps in the Dyck 
paths directly before peak number $i$ in the path, which gives a partition ${\bfc}=(c_1,\dots,c_m)$ of $n$, that is $\sum_ic_i=n$. Let ${\mu_j}$ be the 
number of $c_i$ equal to j $j$. Thus 
$\mu$ (or sometimes written $[1^{\mu_1}2^{\mu_2}\dots n^{\mu_n}]$)  is the {\bf type} of the composition $\bfc$ and of the Dyck path. The number of Dyck paths with $m$ peaks and
of type $\mu$ is known to be the Kreweras number $\krew(n,m,\mu)=\frac{n(n-1)\dots(n-m+1)}{\prod_j \mu_j!}$, from \cite{KrewerasNumbers}. 

In the bijection $\bij$, a tableau with $m$ elements in the top row will be mapped to a
321-avoiding permutation with $m$ right-to-left minima. As discussed above, we can draw a lattice path under these in the permutation matrix and by rotating half a turn obtain 
a bijection to Dyck paths with $m$ peaks.  
The first peak will have height $n$ minus the first right-to-left minima. The distance between two consecutive elements in the top row is mapped to the number of upsteps $c_i$ of the Dyck path for all the remaining peaks. This proves the following corollary.

\begin{corollary}    
\label{cor:321_bij2} For any $b,k\ge 1$ we have the following refinements:
\begin{equation}
\left|\left\{T \in \bigsqcup_{2b+k = n} \SYT^{+k}(\rect{2}{b}) \st \text{ $T$ has $m$ elts in top row}\right\}\right| = \frac{1}{m}\binom{n-1}{m-1}\binom{n-2}{m-1}
\end{equation}
\begin{equation}
\left|\left\{T \in \bigsqcup_{2b+k = n} \SYT^{+k}(\rect{2}{b}) \st \text{ $T$ has elts $a_1, \dots, a_m$ in top row}\right\}\right| = \krew(n,m,\mu),
\end{equation}
where $\mu$ has type $(c_1,\dots, c_m)$ given by $c_i = a_{i+1} - a_{i}$ with $a_{m+1} \coloneqq n$.
\end{corollary} 
Due to possible fixed points, a refinement according to the elements in the bottom row is more difficult.

The bijection $\bij$ also implies the following.

\begin{corollary}
    \label{cor:More shapes} For $n\ge 3$
\begin{align}
\left|\bigsqcup_{2b+k+1 = n} \SYT^{+k}(b+1,b)\right|&=\cat(n)-\cat(n-1)=\frac{3}{n+1}\binom{2n-2}{n}\\
\left|\bigsqcup_{2b+k = n} \SYT^{+k}(b+1,b)/(1)\right|&=\cat(n)-2\cat(n-1)+\cat(n-2)
\end{align}
\end{corollary} 
\begin{proof}
   We can always add $n$ to the bottom right box of a valid $\SYT^{+k}(b,b)$, where $2b+k=n-1$. Thus there are $\cat(n-2)$ set valued tableaux in which $n$ is not the only element in its box. This gives the first statement.

   There is a simple bijection between $\SYT^{+k}(b+1,b)$ and $\SYT^{+k}(b+1,b+1)/(1)$ by rotating the tableau a half turn and replacing each number $i$ with $2b+1+k+1-i$. Now, we can again subtract those where $n+1$ is not alone in its box. This gives $\cat(n)-\cat(n-1)-(\cat(n-1)-\cat(n-2))$ and the second statement follows.
     
\end{proof}

\section{Enumeration According to Peaks}

In \cite[Corollary 5.4]{anderson2017kclasses}, Anderson, Chen, and Tarasca give a formula for the Euler characteristic of a certain \emph{Brill--Noether space}, which they had earlier shown to be equal to the (signed) count of a certain class of set-valued tableaux. Specializing to the two row case and translating into our notation, their formula becomes:
\begin{equation}\label{ACTFormula}
\#\{\SYT^{+k}(2\times b)\} = \frac{1}{k!}\sum_{c=0}^{\lfloor\frac{k}{2}\rfloor}f^{(k-c,c)}f^{(b+k-c,b+c)}(b+k-c-1)_{k-c}(b+c-2)_{c},
\end{equation}
where $f^{\lambda}$ is the number of SYT of shape $\lambda$, and $(x)_a$ is the falling factorial $x(x-1)\cdots(x-a+1)$.

Let $a_{n,k}^{\mathrm{pk}}(321)$ be the number of $321$-avoiding permutations of $n=2b+k$ with exactly $k$ inner peaks (or inner valleys), we get the following result from our bijection above.

\begin{corollary}\label{321PeaksCor}
$$a_{n,k}^{\mathrm{pk}}(321) = \sum_{c=0}^{\lfloor\frac{k}{2}\rfloor}\frac{(k-2c+1)^2}{(k-c+1)(b+k-c+1)}\binom{b+c-2}{c}\binom{b+k-c-1}{b-1}\binom{2b+k}{b+c}.$$ 
\end{corollary}
\begin{proof}
    Use the hook length formula to rewrite \eqref{ACTFormula} 
\begin{align}
\#\{\SYT&^{+k}(2\times b)\}\nonumber \\ &= \frac{1}{k!}\sum_{c=0}^{\lfloor\frac{k}{2}\rfloor}\frac{k!(k-2c+1)}{c!(k-c+1)!}\frac{(2b+k)!(k-2c+1)}{(b+c)!(b+k-c+1)!}\frac{(b+k-c-1)!}{(b-1)!}\frac{(b+c-2)!}{(b-2)!}, \nonumber
\end{align}
which simplifies to the claimed sum.
\end{proof}
This sort of closed form expression is a somewhat pleasant surprise; in \cite[Theorem 3]{Bukata2019DistributionsOS}, the authors give the following generating function as the central result of that paper.
\begin{equation}\label{PeaksGenFun}
\sum_{n\geq 0}\sum_{k\geq 0}a_{n,k}^{\mathrm{pk}}(321)q^kz^n = 1 + z\left(-\frac{-1 + \sqrt{-4z^2q+4z^2-4z+1}}{2z(zq-z+1)}\right)^2,\end{equation}
It is not at all obvious how one would obtain \cref{321PeaksCor} from \eqref{PeaksGenFun} (or vice-versa) using only elementary techniques.

\section{Motzkinlike and Ballotlike Paths}

In addition to $321$-avoiding permutations, we can interpret the $\SYT^{+k}(2\times b)$ in terms of a certain class of bicolored Motzkin paths. 
We recall that a \emph{Motzkin path} of length $n$ is a lattice path in $\Z^2$ from~$(0,0)$ to $(n,0)$ consisting of \emph{up steps} $U=(1,1)$, \emph{down steps} $D=(1,-1)$, and \emph{horizontal steps} $H=(1,0)$ in some order, with the property that the path never goes below the $x$-axis. 

We will color the horizontal steps of the Motzkin paths with $u$ (for upstairs or umber) and $d$ (downstairs or denim). We will consider the following two restrictions on the coloring:
\begin{itemize}
\item [(1)] umber horizontal steps do not occur at height zero;
\item [(2)] denim horizontal steps do not occur before the first down step.
\end{itemize}
We use $\motz(n)$ to denote the set of bicolored Motzkin paths of length $n$, and $\motzE(n)$, $\motzT(n)$, and $\motzET(n)$ to denote the set of paths which satisfies the restrictions (1), (2), and both (1) and (2) respectively.

Extending the well-known bijection between two-rowed rectangular SYTs and Dyck paths, we have the following bijection. A minimal entry is an entry that is smallest in its box.

\begin{thm} \label{thm:motzkin_bij}
There is a bijection $\bijTwo$ between $\SYT^{+k}(\rect{2}{b})$ and $\motzET(2b+k)$ with $k$ horizontal steps. A tableau $S$ maps to the path $\Gamma$ for which:
\begin{itemize} 
\item up steps of $\Gamma$ occur at the minimal entries of boxes in the first row of $S$;
\item down-steps of $\Gamma$ occur at the minimal entries in the second row; 
\item we color a horizontal step of $\Gamma$ umber if its index is a (non-minimal) entry of a box in the first row of $S$, and denim if its index is a (non-minimal) entry in the second row.
\end{itemize}
\end{thm}

\begin{proof}
    
When the Motzkin path is on the x-axis there have been as many upsteps $U$ as downsteps $D$. Thus there cannot be a non-minimal entry in the upper row at this point, and hence no umber colored steps on the x-axis. The second restriction is due the fact that we cannot have a non-minimal element in the bottom row before we have had a minimal element in the first box, which precisely corresponds to having no denim colored step before the first down step. These two restrictions are the only restrictions and constructing the inverse of $\bijTwo$ is straightforward. This shows $\bijTwo$ is a bijection.
\end{proof}

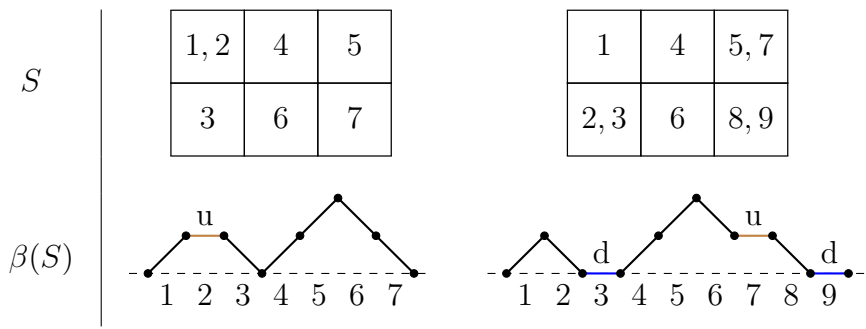
\begin{figure}[hb]
\begin{center}
\begin{tabular}{c|cc}
\begin{tabular}{c}$S$\end{tabular} & \begin{tabular}{c}{\begin{ytableau} 1,2 & 4 & 5 \\ 3 & 6 & 7 \end{ytableau}}\end{tabular} & \begin{tabular}{c}{\begin{ytableau} 1 & 4 & 5,7 \\ 2,3 & 6 & 8,9 \end{ytableau}}\end{tabular}\\ \\
\rule{0pt}{4ex}
\rule{0pt}{4ex}
\begin{tabular}{c}$\bijTwo(S)$\end{tabular} &\begin{tabular}{c}$
\begin{tikzpicture}[scale=0.5]
\draw[dashed] (-0.5,0) -- (7.5,0);
\draw[thick] (0,0) -- (1,1);
\draw[thick, brown] (1,1)--(2,1);
\node[above] at (1.5,1) {u}; 
\draw[thick] (2,1) -- (3,0) -- (4,1) -- (5,2) -- (6,1) -- (7,0);
\draw[fill] (0,0) circle (0.1);
\draw[fill] (1,1) circle (0.1);
\draw[fill] (2,1) circle (0.1);
\draw[fill] (3,0) circle (0.1);
\draw[fill] (4,1) circle (0.1);
\draw[fill] (5,2) circle (0.1);
\draw[fill] (6,1) circle (0.1);
\draw[fill] (7,0) circle (0.1);
\node[below] at (0.5,0) {1}; 
\node[below] at (1.5,0) {2}; 
\node[below] at (2.5,0) {3}; 
\node[below] at (3.5,0) {4}; 
\node[below] at (4.5,0) {5}; 
\node[below] at (5.5,0) {6}; 
\node[below] at (6.5,0) {7}; 
\end{tikzpicture}
$\end{tabular}
& \begin{tabular}{c}$
\begin{tikzpicture}[scale=0.5]
\draw[dashed] (-0.5,0) -- (9.5,0);
\draw[thick] (0,0) -- (1,1) -- (2,0);
\draw[thick,blue] (2,0) -- (3,0);
\node[above] at (2.5,0) {d}; 
\draw[thick] (3,0) -- (4,1) -- (5,2) -- (6,1);
\draw[thick,brown] (6,1) -- (7,1);
\node[above] at (6.5,1) {u}; 
\draw[thick] (7,1) -- (8,0);
\draw[thick,blue] (8,0) -- (9,0);
\node[above] at (8.5,0) {d}; 
\draw[fill] (0,0) circle (0.1);
\draw[fill] (1,1) circle (0.1);
\draw[fill] (2,0) circle (0.1);
\draw[fill] (3,0) circle (0.1);
\draw[fill] (4,1) circle (0.1);
\draw[fill] (5,2) circle (0.1);
\draw[fill] (6,1) circle (0.1);
\draw[fill] (7,1) circle (0.1);
\draw[fill] (8,0) circle (0.1);
\draw[fill] (9,0) circle (0.1);
\node[below] at (0.5,0) {1}; 
\node[below] at (1.5,0) {2}; 
\node[below] at (2.5,0) {3}; 
\node[below] at (3.5,0) {4}; 
\node[below] at (4.5,0) {5}; 
\node[below] at (5.5,0) {6}; 
\node[below] at (6.5,0) {7}; 
\node[below] at (7.5,0) {8}; 
\node[below] at (8.5,0) {9}; 
\end{tikzpicture}
$\end{tabular}
\end{tabular}
\end{center}
\caption{Examples of the bijection $\bijTwo$ between two-rowed rectangular set-valued SYTs and birestricted bicolored Motzkin paths.} \label{fig:motzkin}
\end{figure}
The first two equalities in the following proposition are well-known but the other two seem to be new. We construct a bijection that, when concatenated with $\bijTwo$ in Theorem \ref{thm:motzkin_bij}, gives us a second bijective proof of \eqref{eq:sv_syt_cat}.

\begin{proposition}\label{MotzProp} The Catalan numbers enumerate all four possible restriction on Motzkin paths.
\begin{align} \label{eq:motzT}
|\motz(n)|=&\cat(n+1)\\
|\motzE(n)|=&\cat(n)\\
|\motzT(n)|=&\cat(n)\\
|\motzET(n)|=&\cat(n-1), \text{ for }n\ge 2.
\end{align}
\end{proposition}
 \begin{proof}
 Let us first repeat how to prove that $|\motz(n)|=\cat(n+1)$. The classical argument maps a path $\Gamma\in\motz(n)$ to a Dyck path of length $2n+2$ by starting with an upstep U, then doubling each step of $\Gamma$; each U of $\Gamma $ becomes a UU, each D becomes a DD, and umber and denim horizontal steps become DU and UD respectively. Finally, the Dyck path ends with a downstep. The fact that $|\motzE(n)|=\cat(n)$ is mentioned e.g. in \cite{stanley2015catalan} and we can  modify the standard bijection by just ignoring the starting U and ending D. Since horizontal steps on the x-axis cannot be colored umber, the Dyck path created will not go below the x-axis either. In both cases it is immediate how to construct the inverse.

 Let us now define a map $\phi$ from $\motzT(n)$ to  $\motz(n-1)$ as follows. First we map the path with $n$ horizontal steps labelled $u$ to the path with $n-1$ horizontal steps labelled $u$. All other paths $\Gamma\in\motzT(n)$ contain a downstep and we let $D^i$ be the first downstep in $\Gamma$. By the constraint (2), we know that the the step at position $i-1$ must be a $U$ or an horizontal step labeled $u$. Then we map the the steps in position $i-1,i$ to just one step:
\begin{center}
 \begin{tikzpicture}[scale=0.5]
 \draw[thick,brown] (0,1) -- (1,1);
 \node[above] at (0.5,0) {u}; 
 \draw[thick] (1,1) -- (2,0);
 \node at (3,0.5) {$\overset{\phi}\mapsto$};
 \draw[thick] (4,1) -- (5,0);
 \draw[thick] (0,2) -- (1,3)--(2,2);
 \node at (3,2.5) {$\overset{\phi}\mapsto$};
 \draw[thick, blue] (4,2) -- (5,2);
 \node[above] at (4.5,2) {d}; 
 \draw[fill] (0,1) circle (0.1);
 \draw[fill] (1,1) circle (0.1);
 \draw[fill] (2,0) circle (0.1);
 \draw[fill] (2,2) circle (0.1);
 \draw[fill] (1,3) circle (0.1);
 \draw[fill] (0,2) circle (0.1);
 \draw[fill] (4,1) circle (0.1);
 \draw[fill] (5,2) circle (0.1);
 \draw[fill] (5,0) circle (0.1);
 \draw[fill] (4,2) circle (0.1);
 \end{tikzpicture}
\end{center}
  Since $\Gamma$ is a Motzkin path we know that $i\ge 2$ and the map is well defined. The inverse is easy to define. We find the first occurrence of $D$ or horizontal step labeled denim in a path in $\motz(n-1)$ and expand it to two steps with by applying the reverse of the map above. This is the inverse exactly because paths in $\motzT(n)$ satisfy the constraint (2). The same map $\phi$ works equally well from $\motzET(n)$ to  $\motzE(n-1)$. We conclude:  

\smallskip\noindent
\emph{
The map $\phi$ is a bijection between $\motzT(n)$ and $\motz(n-1)$ and between from $\motzET(n)$ to  $\motzT(n-1)$. In particular
 $|\motzT(n)|=\cat(n)$ and $ |\motzET(n)|=\cat(n-1)$.
}

 Note that the number of horizontal steps is not preserved by $\phi$ but the sum of upsteps U and horizontal umber steps is decreased by one by $\phi$.
 \end{proof}

\subsection{Ballotlike paths}
We can also consider a larger class of paths which we call \emph{ballotlike}. A ballotlike path $P$ is a lattice path in the $1$st quadrant starting at $(0,0)$ and ending at $(n,i)$ which uses the steps $U = (1,1)$, $D = (1,-1)$, $u = (1,0)^{\mathrm{umber}}$ and $d = (1,0)^{\mathrm{denim}}$, subject to the same conditions on $u$ and $d$ steps from the definition of $\motzET(n)$. We write $\bal^*(n,i)$ for the set of ballotlike paths ending at $(n,i)$. The count of such paths turns out to be the sum of a classical ballot number and a binomial coefficient. 

\begin{thm}\label{BallotThm}
For any $(n,i)$ with $0\leq i \leq n$, we have
$$\left|\bal^*(n,i)\right| = \binom{2n-2}{n-i-1} - \binom{2n - 2}{n-i-2} + \binom{n-2}{n-i}.$$

Moreover, if we take the obvious extension of the bijection between $\motzET(n)$ and set-valued SYT of shape $2\times b$, we have for any $(n,i)$ with $0\leq i \leq n$,
$$\left| \bigsqcup_{2b+k-i = n} \SYT^{+k}(b,b-i)\right| = \binom{2n-2}{n-i-1} - \binom{2n - 2}{n-i-2} + \binom{n-2}{n-i}.$$
\end{thm}
\begin{example}\label{Thm13} When $n=4$ and $i=2$ we have 6 set-valued SYT.
\begin{center}
\begin{tabular}{ccc}
    $\begin{array}{|c|c|c|}
    \hline
    1 & 2 & 4\\
    \hline
    3 \\
    \cline{1-1}
    \end{array}$
&
    $\begin{array}{|c|c|c|}
    \hline
    1 & 3 & 4\\
    \hline
    2 \\ 
    \cline{1-1}
    \end{array}$ 
&
    $\begin{array}{|c|c|c|}
    \hline
    1 & 2 & 3\\
    \hline
    4\\
    \cline{1-1}
    \end{array}$
\\
    $\begin{array}{|c|c|c|}
    \hline
    \cline{1-1}
    \end{array}$\\
    $\begin{array}{|c|c|}
    \hline
    1,2,3 & 4 \\
    \hline
    \end{array}$
&
    $\begin{array}{|c|c|}
    \hline
    1,2 & 3,4 \\
    \hline
    \end{array}$
&
    $\begin{array}{|c|c|}
    \hline
    1& 2,3,4 \\
    \hline
    \end{array}$
\end{tabular}
\end{center}
\end{example}

To prove Theorem \ref{BallotThm}, we first let $e_{n,i}$ be the number of ballotlike paths from $(0,0)$ to $(n,i)$ without a $D$ step, and by restriction (2) no $d$ steps either. Let $f_{n,i}$ be the number of ballotlike paths from $(0,0)$ to $(n,i)$ with at least one D step. Since the paths counted by $e_{n,i}$ only uses $U$ and $u$ and the first step cannot be $u$ by restriction (1) it it follows that $e_{n,i}=\binom{n-1}{i-1}$ for $n,i\ge 1$ and $e_{n,0}=0$ for $n\ge 1$. If there has already been a $D$ step in a ballotlike path we can, when $i\ge 1$, come to $(n,i)$ by any of the four steps. For $f_{n,i}, n> i\ge 1$ we thus get the recursion

\[
f_{n,i}=f_{n-1,i-1}+2f_{n-1,i}+f_{n-1,i+1}+e_{n-1,i+1}.
\]
For $f_{n,0}, n\ge 2$ we similarly get the recursion

\[
f_{n,0}=f_{n-1,0}+f_{n-1,1}+e_{n-1,1}.
\]
The problem then reduces to the following formula.
\begin{lemma}
    For $n\ge i\ge 0$ we have $f_{n,i}=\binom{2n-2}{n-i-1} - \binom{2n - 2}{n-i-2} - \binom{n-2}{n-i-1}$.
\end{lemma} 
    \begin{proof}
        Having found the formula it is direct to prove it from the above recursions and the values of $e_{n,i}$.
    \end{proof}

    \noindent
    \emph{Proof of Theorem \ref{BallotThm}.}
        We get the formula by $\#\bal^*(n,i)=e_{n,i}+f_{n,i}$. The second statement follows from bijection 
       $\bijTwo$ since a path ending at height $i$ corresponds to a two-rowed set-valued SYT with $i$ boxes more in the top row. \qed

       Note that this gives a second proof of $|\motzET(n)|=\cat(n-1)$ in Proposition \ref{MotzProp} and a third proof of \eqref{eq:sv_syt_cat}.
 \begin{figure}[ht]
 \begin{center}
 \begin{tabular}{>{\bfseries}c|ccccccccc}
 i/n & \textbf{0} & \textbf{1} & \textbf{2} & \textbf{3} & \textbf{4} & \textbf{5} & \textbf{6} & \textbf{7} & 
 \textbf{8}\\
  \hline 
 8 & \;& \;& \;& \;& \;& \; & \; & \;  & $1,0$ \\
 7 & \;& \;& \;& \;& \;& \; & \; & $1,0$  &$7,0$\\   
 6 & \;& \;& \;& \;& \;& \; & $1,0$  &$6,0$     &$21,7$  \\   
 5 & \;& \;& \;& \;& \;& $1,0$   & $5,0$  & $15,6$  & $35,62$\\    
 4 & \;& \;& \;& \;& $1,0$  & $4,0$  & $10,5$ & $20,44$  & $35,253$\\    
 3 & \;& \;& \;& $1,0$   & $3,0$ & $6,4$ & $10,29$ & $15,144$ & $21,622$\\
 2 & \;& \;& $1,0$ & $2,0$ & $3,3$ & $4,17$ & $5,71$ & $6,270$ &$7,995$\\    
 1 & \;& $1,0$ & $1,0$ &$1,2$& $1,8$ & $1,27$ &$1,89$  & $1,296$ & $1,1000$\\    
 0 & $1,0$ & $0,0$ & $0,1$ & $0,2$ & $0,5$ & $0,14$ & $0,42$ & $0,132$ & $0,429$\\
 \hline
 $\sum$ & $1,0$ & $1,0$ & $2,1$ & $4,4$ & $8,16$ & $16,62$ & $32,236$ & $64,892$ & $128,3368$\\
 \end{tabular}
 \end{center}
 \caption{The values of  $e_{n,i}, f_{n,i}$ for $0\leq i \leq n \leq 8$. The size of $\bal^*(n,i)$ is $e_{n,i}+f_{n,i}$. The bottom row gives the column sums, that is the total number of ballotlike paths with $n$ steps without or with a $D$ step. Formulas for the sums, $n\ge 2$, are $2^{n-1}$ and $\binom{2n-2}{n-1}-2^{n-2}$ respectively.}
 \end{figure}

\section{Generating Functions}\label{Sec:GFun}

In this section, we use generating function techniques to derive enumerative results about the large-scale structure of paths in $\motz$, $\motzE$, $\motzT$, and $\motzET$. We follow the notation of \cite[Chapter I]{FlajoletSedgewick}. We also would like to thank Pierre Bonnet and Guillaume Chapuy for enlightening discussions during FPSAC 2024 in which they suggested several of the techniques used in this section.

Let 
$$E_{n}(U,D,u,d)\coloneqq \sum_{P \in \motz(n)}U^{\#\{\text{up steps of $P$}\}}D^{\#\{\text{down steps of $P$}\}}u^{\#\{\text{umber steps of $P$}\}}d^{\#\{\text{denim steps of $P$}\}},$$
and define the generating function
$$\textbf{E} \coloneqq \sum_{n\geq 0}E_n(U,D,u,d)t^n.$$

Similarly, we define $E_n^{\{1\}}$, $E_n^{\{2\}}$, and $E_n^{\{1,2\}}$ to be the generating polynomials of $\motzE(n)$, $\motzT(n)$, and $\motzET(n)$, respectively, and let $\textbf{E}^{\{1\}}$, $\textbf{E}^{\{2\}}$, and $\textbf{E}^{\{1,2\}}$ be their generating functions.

\begin{thm}
The generating functions $\textbf{E}$, $\textbf{E}^{\{1\}}$, $\textbf{E}^{\{2\}}$, and $\textbf{E}^{\{1,2\}}$ satisfy:
\begin{equation}
\textbf{E} = 1 + (u+d)t\textbf{E} + UDt^2\textbf{E}^2
\end{equation}
\begin{equation}
\textbf{E}^{\{1\}} = \frac{1}{1 - (UDt^2\textbf{E}+dt)}
\end{equation}
\begin{equation}
\textbf{E}^{\{2\}} = \frac{UDt^2\textbf{E}}{1-(UDt^2\textbf{E}+ut)}
\end{equation}
\begin{equation}
\textbf{E}^{\{1,2\}} = \frac{UDt^2}{\left(1-(ut + UDt^2\textbf{E})\right)\left(1 - (dt + UDt^2\textbf{E})\right)}
\end{equation}
\end{thm}

\begin{proof}
We prove each of the functional equations by recursive decompositions of the corresponding paths. See Figure \ref{fig:genfunction} for schematic decompositions of the paths used in the proof.

\begin{itemize}
\item[$\textbf{E}$:] Any $P \in \motz(n)$ is either empty (when $n = 0$) or has one of the following forms (written as words):
\begin{itemize}
\item $uP'$ or $dP'$
\item $UP'DP''$,
\end{itemize}
where $P'$ and $P''$ are (generic) shorter bicolored Motzkin paths.
Summing over all $n$, this translates to
$$\textbf{E} = 1 + ut\textbf{E} + dt\textbf{E} + Ut\textbf{E}Dt\textbf{E},$$
which easily simplifies to our desired formula.

\item[$\textbf{E}^{\{1\}}$:] Any $P \in \motzE(n)$ can be written as a sequence of $d$-steps and subpaths of the form $UP'D$, where $P' \in \motz(m)$ for some $m< n$. Using the notation from \cite[Theorem I.1]{FlajoletSedgewick}, this means that
\[
\textbf{E}^{\{1\}} = \mathsf{SEQ}\left(Ut\textbf{E}Dt, dt\right)
= \frac{1}{1-\left(UDt^2\textbf{E} + dt\right)},
\]
as desired.

\item[$\textbf{E}^{\{1,2\}}$:] Any $P\in\motzET(n)$ begins with an initial $U$ step, a sequence of $U$ and $u$ steps, and a first down step $D$. Suppose that this first $D$ step ends at height $i \geq 0$. The remainder of $P$ is of the form $DQ_1DQ_2\cdots DQ_{i}DQ_{i+1}$, where $Q_1,\dots,Q_i$ belong to $\motz(m_1),\dots,\motz(m_i)$, respectively, and $Q_{i+1} \in \motzE(m_{i+1})$.

For any fixed $i$, we sum this decomposition over all $n$. The initial segment of $P$ can be written as:
\begin{itemize}
\item An initial $U$ step,
\item A sequence of $U$ and $u$ steps of indeterminate length, containing exactly $i$ up steps,
\item A $D$ step ending at height $i$.
\end{itemize}
Hence, this initial segment contributes a factor of $UDt^2[x^{i}]\frac{1}{1-(Utx+ut)}$ to the $i$th summand of $\textbf{E}^{\{1,2\}}$. By the general theory of geometric series, we know that $[x^{i}]\frac{1}{1-(Utx+ut)}$ is $\displaystyle\sum_{k\geq i}\binom{k}{i}(Ut)^i(ut)^{k-i}$. By reindexing, this becomes $(Ut)^i\displaystyle\sum_{k\geq 0}\binom{k+i}{k}(ut)^k = \frac{(Ut)^i}{(1-ut)^{i+1}}$. Thus, the initial segment of contributes a total factor of $\frac{UDt^2(Ut)^i}{(1-ut)^{i+1}}$ to the $i$th summand of $\textbf{E}^{\{1,2\}}$. On the other hand, when we sum over all $n$, the remainder of this generic path contributes a factor of $(Dt\textbf{E})^i\textbf{E}^{\{1\}}$ to the $i$th summand of $\textbf{E}^{\{1,2\}}$.

Putting all of this information together, we have
\begin{align*}\textbf{E}^{\{1,2\}} &= \sum_{i\geq 0}\frac{UDt^2(Ut)^i(Dt\textbf{E})^i\textbf{E}^{\{1\}}}{(1-ut)^{i+1}}\\
&= \frac{UDt^2\textbf{E}^{\{1\}}}{1-ut}\sum_{i\geq 0}\left(\frac{UDt^2\textbf{E}}{1-ut}\right)^i\\
&=\frac{UDt^2\textbf{E}^{\{1\}}}{1-ut} \cdot \frac{1}{1- \frac{UDt^2\textbf{E}}{1-ut}}\\
&= \frac{UDt^2\textbf{E}^{\{1\}}}{1-ut-UDt^2\textbf{E}}.\end{align*}

Substituting in $\textbf{E}^{\{1\}} = \frac{1}{1- (UDt^2\textbf{E}+dt)}$ yields the desired functional equation.

\item[$\textbf{E}^{\{2\}}$:] This case is similar to the case of $\textbf{E}^{\{1,2\}}$, except $Q_{i+1}$ is a path in $\motz(m_{i+1})$ instead of in $\motzE(m_{i+1})$.

\end{itemize}
\end{proof}

\begin{figure}[ht]\label{fig:genfunction}
\begin{tabular}{cc}
\begin{tikzpicture}[scale = 0.5, decoration = {snake, amplitude = 0.2mm}]
  \draw{ (0,0) -- (1,1)};
  \draw{(1,1)--(5,1)};
  \draw[decorate]{(1,1).. controls (2,3) and (3,3) .. (5,1)};
  \draw{(5,1)--(6,0)};
  \node at (2.75,1.75) {\textbf{E}};
  \draw{(6,0)--(10,0)};
  \draw[decorate]{(6,0).. controls (7,2) and (8,2) .. (10,0)};
  \node at (7.75,0.75) {\textbf{E}};
\end{tikzpicture}
&
\begin{tikzpicture}[scale = 0.5, decoration = {snake, amplitude = 0.2mm}]
\draw{(0,0) -- (1,1)};
\draw[decorate]{(1,1)--(3,4)};
\draw{(3,4)--(4,3)};
\draw{(4,3)--(6,3)};
\draw[decorate]{(4,3) .. controls (4.5,4) and (5,4) .. (6,3)};
\node at (5,3.4) {\textbf{E}};
\draw{(6,3)--(7,2)};
\draw{(7,2)--(9,2)};
\draw[decorate]{(7,2) .. controls (7.5,3) and (8,3) .. (9,2)};
\node at (8,2.4) {\textbf{E}};
\draw{(9,2)--(10,1)};
\draw{(10,1)--(12,1)};
\draw[decorate]{(10,1) .. controls (10.5,2) and (11,2) .. (12,1)};
\node at (11,1.4) {\textbf{E}};
\draw{(12,1)--(13,0)};
\draw{(13,0)--(15,0)};
\draw[decorate]{(13,0) .. controls (13.25,1.5) and (14.3,1.5) .. (15,0)};
\node at (14,0.4) {${\textbf{E}}^{\{1\}}$};
\end{tikzpicture}
\end{tabular}
\caption{Schematic decompositions of paths enumerated by $\textbf{E}$ (left) and by $\textbf{E}^{\{1,2\}}$ (right).}
\end{figure}
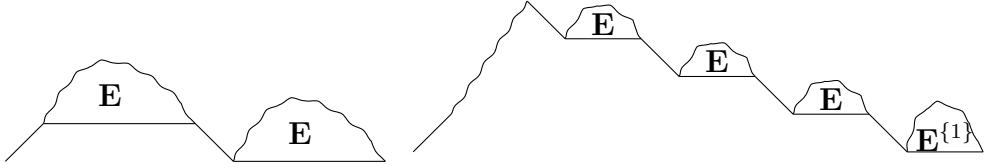

\begin{corollary}
We have $\textbf{E} = \frac{1 - (u+d)t - \sqrt{((u+d)t-1)^2 - 4UDt^2}}{2UDt^2}$.
\end{corollary}

\begin{proof}
The quadratic formula tells us that $\textbf{E} = \frac{1 - (u+d)t \pm \sqrt{((u+d)t-1)^2 - 4UDt^2}}{2UDt^2}$, and by a Taylor expansion we can verify that only the desired solution has constant term $1$.
\end{proof}

Via some further algebraic manipulation, we thus have the following.

\begin{corollary}
We have
\begin{equation}
\textbf{E}^{\{1\}} = \frac{2}{1 + (u-d)t + \sqrt{((u+d)t-1)^2-4UDt^2}}
\end{equation}
\begin{equation}
\textbf{E}^{\{2\}} = \frac{1 - (u+d)t - \sqrt{((u+d)t-1)^2-4UDt^2}}{1 - (u-d)t + \sqrt{((u+d)t-1)^2-4UDt^2}}
\end{equation}
\begin{equation}\textbf{E}^{\{1,2\}} = \frac{2UDt^2}{1-(u+d)t + 2(ud-UD)t^2 + \sqrt{((u+d)t - 1)^2 - 4UDt^2}}.\end{equation}
\end{corollary}

\begin{example}
The first few terms of the Taylor series for $\textbf{E}^{\{1,2\}}$ are:
$$UDt^2 + U(u+d)Dt^3 +U\left(d^2+ud+2UD+u^2\right)Dt^4+\cdots.$$

These correspond to
\begin{align*}\motzET(2) &= \{UD\}\\ 
\motzET(3) &= \{UuD, UDd\}\\
\motzET(4) &= \{UDdd, UuDd, UDUD, UUDD, UuuD\}.
\end{align*}
\end{example}

These sorts of generating function techniques provide surprising utility. For example, consider the generating function $\textbf{F}(U,t)\coloneqq \left.\textbf{E}^{\{1,2\}}\right|_{u=d=D=1}$. The coefficient of $t^n$ in $\textbf{F}$ is a polynomial in $U$ wherein a term of the form $aU^k$ means that there are exactly $a$ paths in $\motzET(n)$ with $k$ up-steps. This means that
$$\left.\frac{\partial \textbf{F}}{\partial U}\right|_{U=1} = \sum_{n\geq 1} \left(\sum_{k\geq 0}k\cdot\#\{P \in \motzET(n) \st P \text{ has $k$ up-steps}\}\right)t^n,$$
so $$\frac{[t^n]\left.\frac{\partial \textbf{F}}{\partial U}\right|_{U=1}}{[t^n]\textbf{F}(1,t)} = \frac{\sum_{k\geq 0}k\cdot\#\{P \in \motzET(n) \st P \text{ has $k$ up-steps}\}}{\#\motzET(n)} = \mathbb{E}_{\text{unif}}\left(\#\{\text{up-steps of $P$}\}\right).$$

One can check that
$$\left.\frac{\partial \textbf{F}}{\partial U}\right|_{U=1} = \frac{1 - \sqrt{1 - 4 t} - 2t(4 - 3\sqrt{1 - 4 t} + t(-9 + 3\sqrt{1 - 4t} + 4t))}{2-8t}.$$

By computing the Taylor expansion\footnote{In this case, using \emph{Mathematica}.} one finds that 
$$\left.\frac{\partial \textbf{F}}{\partial U}\right|_{U=1} = t^2 + \sum_{n\geq3}\frac{(n^2+n-6)(2n-4)!}{n!(n-2)!}t^n.$$

Meanwhile, we know that $\#\motzET(n) = \cat(n-1)$ for all $n\geq 2$, so we immediately have that 
$$\mathbb{E}_{\text{unif}}\left(\#\{\text{up-steps of $P$}\}\right) = \begin{cases}1, & n=2\\ \frac{n^2+n-6}{4n-6}, & n\geq 3\end{cases}.$$

Now, notice that the number of up-steps in a path $P \in \motzET(n)$ is equal to the number of columns in the corresponding two-row set-valued SYT, since under our bijection each up step corresponds to the smallest entry in a cell in the top row of a set-valued SYT.

The preceding discussion provides a proof of the following theorem, which resolves in the affirmative a conjecture in the extended abstract \cite{PlusKFPSAC}.

\begin{thm}
For all $n\geq 3$, if we sample uniformly at random, the expected number of columns of a $T \in \displaystyle\bigsqcup_{2b+k = n}\SYT^{+k}(2\times b)$ is $\frac{n^2+n-6}{4n-6}$.
\end{thm}

In fact, since $\textbf{E}^{\{1,2\}}$ is symmetric under interchanging $U$ with $D$ and under interchanging $u$ with $d$, we can compute the expected numbers of $D$, $u$, and $d$-steps as well. For $n\geq 3$, the expected values of each step type are shown in the following table.

\begin{figure}[h]
\begin{center}
    \begin{tabular}{c||c|c}
    Step & $U/D$ & $u/d$\\
    \hline
    Expected number & $\frac{n^2+n-6}{4n-6}$ & $\frac{n^2-4n+6}{4n-6}$
    \end{tabular}
\end{center}
\end{figure}

\section{Set-Valued Linear Extensions}\label{Poset-Section}
Given the well-behaved numerology for the two-row set-valued SYT, it seems natural to try and find $q$-analogues of our results with respect to some statistic. In the following sections we begin that work using the $\comaj^{+k}$ statistic mentioned in the introduction. 

Many of the technical results about $\comaj^{+k}$ can be stated and proved in the language of linear extensions of partially-ordered sets. Throughout this section, $P$ will denote a finite partially-ordered set, $\mathcal{J}(P)$ will be its poset of order ideals, and $\mathcal{L}(P)$ its set of linear extensions.

A poset $P$ is said to be \emph{naturally-labeled} if it has a \emph{labeling} $\omega: P \to [|P|]$ such that if $p_1 \preceq_P p_2$, then $\omega(p_1) \leq \omega(p_2)$.

In \cite{hopkins2021qenumeration} a \emph{set-valued linear extension} of $P$ is defined as a map $S: P \to 2^{[n+k]}$ such that
\begin{itemize}
    \item $S(p) \neq \emptyset$ for any $p \in P$.
    \item $\bigcup S(p) = [n+k]$.
    \item $S(p) \cap S(q) = \emptyset$ whenever $p\neq q$.
    \item $\max S(p) < \min S(q)$ whenever $p<q$.
\end{itemize}

We write $\Lin^{+k}(P)$ for the set of set-valued linear extensions of $P$ taking values in $[n+k]$.

\begin{proposition}\label{Set_Valued_Lin_Bijection}
For any finite (naturally-labeled) poset $P$ with $n$ elements. There is a bijection between $\Lin^{+k}(P)$ and the set of triples $(T, \overline{t}, \overline{p})$, where $T \in \Lin(P)$, $\overline{t} = (t_1,\dots, t_{k})$ such that $0 < t_1 \leq \cdots \leq t_{k} \leq n$, and $\overline{p} \in P^{k}$ such that $p_i \in \max T^{-1}(\{1,\dots, t_i\})$ for all $i$.
\end{proposition}
\begin{proof}
Given any such triple $(T,\overline{t},\overline{p})$, we construct a set-valued linear extension in $\Lin^{+k}(P)$. For each $i=1, \dots, k$, we define $T_i \in \Lin^{+i}(P)$ in terms of $T_{i-1}$ (taking $T_0 \coloneqq T$) in the following way:

$$T_i(s) = \begin{cases} T_{i-1}(s) \cup \{t_i + i\}, & s = p_i,\\ T_{i-1}(s), & s \in T^{-1}(\{1,\dots, t_i\})\setminus \{p_i\}\\ \{x +1 \st x \in T_{i-1}(s)\}, & s \in T^{-1}(\{t_i+1,\dots, n\})\end{cases}.$$

In other words, the set-valued filling $T_i$ is obtained from the set-valued filling $T_{i-1}$ by adjoining $t_i+i$ to $T_{i-1}(p_i)$ and then increasing each entry of $T_{i-1}(s)$ by $1$ for each $s \in T^{-1}(\{t_i+1,\dots,n\})$. 

At each step, we see that $T_i$ is a set-valued linear extension by induction. Indeed, by assumption $T_0$ is a linear extension. For each $i = 1,\dots, k$, the sets $T_{i-1}(p)$ partition the set $[n+i-1]$, so the sets $T_i(p)$ form a partition of $[n+i]$. 

It thus remains to check that if $p < p'$ then $\max(T_i(p)) < \min(T_i(p'))$. By assumption this holds for $T_{i-1}$. If $T(p) \leq t_i$ and $p \neq p_i$, then $\max(T_{i}(p)) \leq t_i +i$ and if $T(p') > t_i$ then $\min(T_i(p)) > t_i +i$. If $T(p)$ and $T(p')$ are both at most $t_i$, then $p \neq p_i$ so $\max(T_i(p)) = \max(T_{i-1}(p))$ and $\min(T_i(p')) = \min(T_{i-1}(p'))$. Finally, if $T(p)$ and $T(p')$ are both greater than $t_i$, then $\max(T_i(p)) = \max(T_{i-1}(p)) + 1$ and $\min(T_i(p')) = \min(T_{i-1}(p')) + 1$. This completes the induction and shows that $T_i$ belongs to $\Lin^{+i}(P)$ for all $i$.

Conversely, given a $S \in \Lin^{+k}(P)$ we can obtain a triple $(T,\overline{t},\overline{p})$. Indeed, given such a $S$, we construct a sequence $S_{k}, S_{k-1},\dots,S_0$ with $S_{k} \coloneqq S$ and $S_i \in \Lin^{+i}(P)$ for all $i$ by reversing the above procedure. More precisely, let $p_i$ be the $\omega$-largest element of $P$ such that $\#S_i(P)>1$, and let $t_i \coloneqq \max(S_i(p_i)) - i$. We then define $S_{i-1}$ in terms of $S_i$ as:
$$S_{i-1}(s) = \begin{cases} S_{i}(s)\setminus \{\max(S_i(s))\}, & s = p_i,\\ S_{i}(s), & S_i(s) \subseteq \{1,\dots, t_i +i\}, \; s\neq p_i\\ \{x -1 \st x \in S_{i}(s)\}, & S_i(s) \subseteq \{t_i+ i +1,\dots, n+i\}\end{cases}.$$

If we let $T \coloneqq S_0$, $\overline{t}\coloneqq (t_1,\dots,t_{k})$ and $\overline{p} \coloneqq (p_1,\dots,p_{k})$, then $(T,\overline{t},\overline{p})$ is our desired triple. Indeed, by construction we see that $T$ is an (ordinary) linear extension of $P$, that $0\leq t_1 \leq \cdots \leq t_{k} \leq n$, and that $p_i \in \max T^{-1}(\{1,\dots, t_i\})$ for all $i$. Moreover, it is clear that these two constructions are mutually inverse.
\end{proof}

\begin{example}
Let $P(\lambda)$ be the set of cells of the Ferrers diagram for $\lambda$, with $c_1 \preceq_P c_2$ if $c_1$ is weakly north and/or weakly west of $c_2$. Ordering the cells of $\lambda$ left to right from the top to the bottom yields a natural labeling $\omega$ (see Figure~\ref{fig:NatLabel}).

\begin{figure}
\begin{ytableau}
1 & 2 & 3 & 4 \\
5 & 6 & 7 & 8 \\
9 & 10 & 11 & 12
\end{ytableau}
\caption{The natural labeling $\omega$ on $P(3\times 4)$.}\label{fig:NatLabel}
\end{figure}

The set-valued Young tableau from \ref{SVSYTEx} is a set-valued linear extension of $P(3\times 4)$.

\end{example}

Motivated by the techniques in \cite{hopkins2021qenumeration}, we wish to define a probability distribution on the set of multichains of order ideals of $P$. Informally, we will assign a weight to a multichain $\mathcal{I}$ by $q$-counting the set-valued linear extensions ``compatible'' with $\mathcal{I}$ with respect to $\comaj^{+k}$.

More formally, let $S \in \Lin^{+k}(P)$, and let $d_1,\dots,d_{k}$ be the set of all non-minimum entries of $S(p)$ for all $p \in P$, written in increasing order. For $i \in [k+1]$ define $S_i$ to be the restriction of $S$ to $S^{-1}(\{d_{i-1},\dots,d_i-1\})$, where $d_0 = 0$ and $d_{k+1} = n+k +1$. We define a \emph{descent} of $S_i$ to be a $j$ such that $j$ and $j+1$ are in the range of $S_i$ and $S^{-1}(j+1) < S^{-1}(j)$ and write $\Des(S_i)$ for the descent set of $S_i$. We then define
$$\Des^{+k}(S) \coloneqq \bigcup_{i=1}^{k+1}\Des(S_i) \cup \{d_1,\dots,d_{k}\},$$
and
\begin{align*}\comaj^{+k}(S) &\coloneqq \sum_{j \in \Des^{+k}(S)}n+k - j\\
&= \sum_{i=1}^{k+1}\sum_{j \in \Des(S_i)}(n+k - j) + \sum_{i=1}^{k}(n+k - d_{i}).\end{align*}

Now, let $S \in \Lin^{+k}(P)$ correspond to the triple $(T,\overline{t},\overline{p})$. We say that $S$ is \emph{compatible with $\mathcal{I}$} if $T^{-1}(\{1,\dots,t_j\}) = I_j$ for all $j$, and define 
$$\vartheta(T,\overline{t}) \coloneqq q^{\comaj^{+k}(S)}.$$

We then define 
$$ \mu_{\Lin}^q(\mathcal{I}) = \frac{1}{\mathcal{Z}_{\Lin}(q)}\sum_{\substack{T \in \Lin(P)\\ T^{-1}(\{1,\dots,t_j\}) = I_j}}\vartheta(T,\overline{t}),$$
where $\mathcal{Z}_{\Lin}(q)$ is a normalizing constant which we will compute in this section.

The reader will notice that the definition of $\vartheta$ does not use all of the data from the set-valued linear extension $S$, as it does not depend on $\overline{p}$. It is therefore reasonable to ask how many different $S \in \Lin^{+k}(P)$ are compatible with a given $\mathcal{I}$.

Given an order ideal $I$ of $P$, we define the \emph{down-degree} of $I$ to be
$$ \ddeg(I) \coloneqq \#\{\text{maximal elements of $I$}\}.$$
The term ``down-degree'' arises from the fact that $I$ covers exactly $\ddeg(I)$ many elements in the lattice of order ideals $\mathcal{J}(P)$.

\begin{remark}
We observe that for any fixed $T$ and $\bar{t}$, there are exactly $\displaystyle\prod_{j=1}^{k}\ddeg(I_j)$ triples $(T,\bar{t},\bar{p})$ that yield the same multichain of order ideals
$$\mathcal{I} = \emptyset \subseteq I_1 \subseteq \cdots \subseteq I_{k} \subseteq P$$
with $I_j \coloneqq T^{-1}(\{1,\dots,k_j\})$.
\end{remark}

Before proceeding, we need the following technical lemma.

\begin{lemma}[{{\cite[Lemma 2.7]{hopkins2021qenumeration}}}]\label{lem:PermLemma}
For any $X\subseteq \{0,\dots n\}$, let $$\pi(X,t)=\#\{j \in X \st  j < t\}+\begin{cases}n-t, & t \notin X\\ 0, & t\in X\end{cases}.$$ Then, the sequence $\{\pi(X,0),\dots,\pi(X,n)\}$, is a permutation of $\{0,\dots, n\}$.
\end{lemma}

\begin{remark}\label{eq:q-Lin_dist}
We have
\begin{align*}&\vartheta(T,\{t_1,\dots,t_{k}\})\\ &= q^{\binom{k}{2}} \prod_{j \in \Des(T)}q^{n-j}\left(\prod_{i=1}^{k}q^{\begin{cases}n-t_i, & t_i \notin \Des(T)\\ 0, & t_i \in \Des(T)\end{cases}}\right)q^{\displaystyle\sum_{i=1}^{k+1}(k+1-i)\cdot\#\{j \in \Des(T) \st t_{i-1} < j < t_i\}}\\
&=q^{\comaj(T)+\binom{k}{2}}\cdot \left(\prod_{i=1}^{k}q^{\begin{cases}n-t_i, & t_i \notin \Des(T)\\ 0, & t_i \in \Des(T)\end{cases}}\right) \cdot q^{\displaystyle\sum_{i=1}^{k+1}(k+1-i)\cdot\#\{j \in \Des(T) \st t_{i-1} < j < t_i\}}\end{align*}
\end{remark}

\begin{proposition}
\[\sum_{T\in \Lin(P)}\sum_{0\le t_1\le\dots\le t_k} \vartheta(T,\{t_1,\dots,t_{k}\})=q^{\binom{k}{2}}\cdot\qbinom{k+n}{k}_q\cdot \sum_{T\in \Lin(P)} q^{\comaj(T)}
\]
\end{proposition}
\begin{proof}
By Remark \ref{eq:q-Lin_dist} we have, in the notation of Lemma~\ref{lem:PermLemma}, that $$\vartheta(T,\{t_1,\dots,t_{k}\})=q^{\comaj(T)+\binom{k}{2}}\prod_{i=1}^k q^{\pi(D(T),t_i)}.$$ 

For a fixed $T$ we thus get 
\begin{align}
\sum_{0\le t_1\le\dots\le t_k\le n} \vartheta(T,\{t_1,\dots,t_{k}\})&=q^{\comaj(T)+\binom{k}{2}} \sum_{0\le t_1\le\dots\le t_k\le n}\prod_{i=1}^k q^{\pi(D(T),t_i)}
\\
&=q^{\comaj(T)+\binom{k}{2}} \sum_{0\le t_1\le\dots\le t_k\le n}\prod_{i=1}^k q^{t_i},
\end{align}
where the last equality follows from the fact that $\pi(X,t)$ forms a permutation and when summing over all ways of choosing an $\ell$-subset with repetition it does not matter which permutation it is.

We can express $\displaystyle\sum_{0\leq t_1\leq \cdots \leq t_{k}\leq n}q^{t_1+\cdots + t_{k}}$ in terms of reverse $P$-partitions on the rectangular shape $1 \times k$ with entries bounded by $n$. As is well-known (see, e.g., \cite{stanley1999ec2}) we have
$$\sum_{0\leq t_1\leq \cdots \leq t_{k}\leq n}q^{t_1+\cdots + t_{k}} = \sum_{\pi \in \RPP_n(1\times k)}q^{|\pi|} = \qbinom{k+n}{k}_q.$$
\end{proof}

As a consequence, we can now define $\mu_{\Lin}^q$ in terms of linear extensions and the comajor index:

\begin{align*}\mu_{\Lin}^q(\mathcal{I}) &= \frac{1}{q^{\binom{k}{2}}\qbinom{n+k}{n}_q\displaystyle\sum_{T \in \Lin(P)}q^{\comaj(T)}}\sum_{\substack{T \in \Lin(P)\\0\leq t_1\leq\cdots\leq t_k\leq n}}\hspace{-5mm}\vartheta(T,\{t_1,\dots,t_k\})\prod_{i=1}^k\chi(T^{-1}(\{1,\dots,t_i\})=I_i)\\
&= \frac{1}{q^{\binom{k}{2}}\qbinom{n+k}{n}_q\displaystyle\sum_{T \in \Lin(P)}q^{\comaj(T)}}\sum_{\substack{S \in \Lin^{+k}(P)\\ \text{$S$ compatible with $\mathcal{I}$}}}q^{\comaj^{+k}(S)}.\end{align*}

Moreover, we have the following computation of the expectation of the product of down degrees.
\begin{thm}\label{thm:ExpectedComaj}
We have
$$\mathbb{E}_{\mu_{\Lin}^q}\left(\prod_{j=1}^{k}\ddeg(I_j)\right) = \frac{\displaystyle\sum_{S \in \Lin^{+k}(P)}q^{\comaj^{+k}(S)}}{q^{\binom{k}{2}}\qbinom{n+k}{n}_q\cdot\displaystyle\sum_{T \in \Lin(P)}q^{\comaj(T)}}.$$
\end{thm}

\begin{proof}

By definition, we have
\begin{align*}
\mathbb{E}_{\mu_{\Lin}}&\left(\prod_{j=1}^{k}\ddeg(I_j)\right)\\ &= \frac{\displaystyle\sum_{\emptyset\subseteq I_1 \subseteq\cdots \subseteq I_{k}\subseteq P}\sum_{T \in \Lin(P)}\prod_{j=1}^{k}\ddeg(I_j)\cdot\vartheta(T,\{\#I_1,\dots,\#I_{k}\})\cdot\prod_{j=1}^{k}\chi(T^{-1}(\#I_j) = I_j)}{\displaystyle q^{\binom{k}{2}}\qbinom{n+k}{n}_q\cdot\sum_{T \in \Lin(P)}q^{\comaj(T)}}\\
&= \frac{\displaystyle\sum_{S \in \Lin^{+k}(P)}q^{\comaj^{+k}(S)}}{\displaystyle q^{\binom{k}{2}}\qbinom{n+k}{n}_q\cdot\sum_{T \in \Lin(P)}q^{\comaj(T)}},
\end{align*}
with the last equality following from the bijection in \cref{Set_Valued_Lin_Bijection}.
\end{proof}

\begin{remark}
The definitions of $\comaj^{+k}$ and $\mu_{\Lin}^q$ can also be justified probabilistically in terms of set-valued $P$-partitions (due in the plane partition case to Lam and Pylyavskyy \cite{lam2007combinatorial}). In this setting, one can define a probability distribution 
$$\mu_{\leq m}^q(\mathcal{I}) \propto \displaystyle \sum_{\substack{\tau \in \RPP^{+k}_{\leq m}(P) \\ \text{$\tau$ compatible with $\mathcal{I}$}}} q^{|\tau| - k},$$
where $|\tau|$ is the sum of the values of $\tau$. One can then show that $\mu_{\leq m}^q \to \mu_{\Lin}^q$ as $m \to \infty$.
\end{remark}

\section{Equidistribution of Descents}
Another nice property of natural set-valued descents is that they are well-behaved under the conjugation of tableaux.

\begin{thm}\label{thm:Equidistribution}
Let $\lambda$ be a partition of $n$, $k\geq 0$ be a nonnegative integer, and let $X \subseteq [n+k]$. Then the sets
$$\{S \in \SYT^{+k}(\lambda) \st \Des(S) = X\}$$
and
$$\{S \in \SYT^{+k}(\lambda') \st \Des(S) = X\}$$
are in bijection.
\end{thm}

\begin{proof}
Let $\mu_1/\mu_2$ be a skew shape. We recall that as posets, $\mu_1/\mu_2$ and $(\mu_1/\mu_2)'$ are isomorphic and indeed differ only by choosing different natural labelings. From \cite[Section 3.13]{stanley1996ec1}, we know that the number of linear extensions of a poset $P$ with descent set $X$ is determined by the flag $h$-vector of $P$, and hence is independent of the choice of natural labeling of $P$. 

Let $S \in \SYT^{+k}(\lambda)$, and let $S^*$ be the filling of $\lambda$ given by $S^*(c) = \min S(c)$ for all $c$. Let $t_1,\dots, t_k$ be the non-minimal entries of $S$ in increasing order. The descents of $S$ consist of the descents of $S^*$ along with all of the $t_i$. As in the proof of~\cref{Set_Valued_Lin_Bijection}, we can use the $t_i$ to decompose $\lambda$ into a sequence of skew shapes: we define $I_j$ to be $(S^*)^{-1}(\{1,\dots,t_j\})$, and then we can write 
$$\lambda = I_1 \sqcup (I_2\setminus I_1)\sqcup \cdots \sqcup (I_k\setminus I_{k-1}) \sqcup (\lambda \setminus I_k).$$

Now, every descent of $S^*$ belongs to some $(I_j \setminus I_{j-1})$ (taking $I_0 = \emptyset$ and $I_{k+1} = \lambda$) in the sense that the descent top and descent bottom both appear in $S^*(I_j\setminus I_{j-1})$.

So, suppose $\Des(S) = X_1 \sqcup \cdots \sqcup X_{k+1} \sqcup \{t_1,\dots,t_k\}$, where $X_j$ are those descents appearing in $I_j \setminus I_{j-1}$. We can think of $\left.S^*\right|_{I_j\setminus I_{j-1}}$ as a shifted linear extension of $I_j\setminus I_{j-1}$. By our first observation, the (shifted) linear extensions of $I_j\setminus I_{j-1}$ are in bijections with the linear extensions of $(I_j\setminus I_{j-1})'$ (shifted by the same amount). Since this holds for all $j$, we see that the injective maps $S^*: \lambda \to [n+k]$  with descent set $X_1 \sqcup \cdots \sqcup X_{k+1}$ are in bijection with the injective maps $T^*: \lambda' \to [n+k]$ with descent set $X_1 \sqcup \cdots \sqcup X_{k+1}$. To go from $S^*$ to an $S \in \SYT^{+K}(\lambda)$ with descent set $X$, we insert each $t_i$ into a maximal cell of $I_i$ in $\prod \ddeg(I_j)$ ways. Similarly, to go from $T^*$ to a $T \in \SYT^{+k}(\lambda')$, we insert each $t_i$ into a maximal cell of $(I_i)'$ in $\prod \ddeg(I_j)$ ways. This completes the proof. 
\end{proof}
\section{Future Work}
\subsection{q-Catalan and q-Narayana}\label{qCatSect}
Given the well-behaved numerology for $\SYT^{+k}(2\times b)$, it seems natural to consider the following $q$-analogs:
$$\overset{\sim}{\cat}_n(q) \coloneqq \sum_{2b+k = n+1}\left(\sum_{S \in \SYT^{+k}(2\times b)}q^{\comaj^{+k}(S)}\right)$$
and
$$\overset{\sim}{N_{n,m}}(q) \coloneqq \sum_{2b +k = n+1}\left(\sum_{\substack{S \in \SYT^{+k}(2 \times b)\\m\text{ elts in top row}}}q^{\comaj^{+k}(S)}\right).$$

Using the bijection in Theorem \ref{thm:motzkin_bij} and a double recursion we can compute the polynomials $\overset{\sim}{\cat}_q$ and $\overset{\sim} {N_{n,m}}(q)$ for small values of $n,m$. They do not seem to match any statistic we have found in the literature.

\begin{figure}
\begin{center}
\begin{tabular}{c|c}
$n$ & $\overset{\sim}{\cat}_n(q)$\\
\hline
$1$ & $1$\\
\hline
$2$ & $q+1$\\
\hline
$3$ & $q^3 + 2q^2 + q + 1$\\
\hline
$4$ & $q^6 + 2q^5 + 3q^4 + 3q^3 + 2q^2 + 2q + 1$\\
\hline
$5$ & $q^{10} + 2q^9 + 3q^8 + 7q^7 + 6q^6 + 5q^5 + 6q^4 + 7q^3 + 3q^2 + q + 1$
\end{tabular}
\end{center}
\caption{The first six $\overset{\sim}{\cat}_n(q)$ polynomials.}
\end{figure}

\begin{figure}
\begin{center}
\begin{tabular}{>{\bfseries}c|c|c|c|c}
n/m & \textbf{1} & \textbf{2} & \textbf{3} & \textbf{4} \\
\hline
1 & $1$ &  &  &  \\
\hline
2 & $1$ & $q$ & & \\
\hline
3 & $q$ & $2q^2+1$ & $q^3$ &  \\
\hline
4 & $q^3$ & $2q^4 + q^3 + q^2 + q + 1$ & $2q^5 + q^4 + q^3 + q^2 + q$ & $q^6$ \\
\end{tabular}
\end{center}
\caption{$\overset{\sim}{N_{n,m}}(q)$ for $2\leq n \leq 5$.}
\end{figure}

\begin{question}
Are there better formulas for $\overset{\sim}{\cat}_n(q)$ and $\overset{\sim}{N_{n,m}}(q)$?
\end{question}

\subsection{Expected Number of Columns}

\begin{question}
The proofs of the expected value computations in Section~\ref{Sec:GFun} rely on algebraic manipulation of generating functions. Are there bijective proofs?
\end{question}

\begin{question}
Is there a nice formula for the $q$-version? Specifically, is there a better formula for
$$\mathbb{E}_q\left(\text{\# of columns of $S$}\right) = \frac{\displaystyle\sum_{2b+k=n+1}b\cdot\left(\sum_{S \in \SYT^{+k}(2\times b)}q^{\comaj^{+k}(S)}\right)}{\overset{\sim}{\cat_n}(q)}?$$
\end{question}
\subsection{Equidistribution of descents}

\begin{question}
    Is there a bijective proof of Theorem \ref{thm:Equidistribution}? 
\end{question}
    For skew shapes with two rows there is a reasonably easy bijection. One can fix all descent tops and descent bottoms, then there is a unique way to fill in the other elements without creating any other descents. But for skew shapes with more rows this is no longer true and it would be interesting to see a bijective proof of the equidistribution, that could hopefully shed some more light on this property.

\printbibliography

\end{document}